\documentclass[12pt]{amsart}
\usepackage[margin=1in,includeheadfoot]{geometry}
\usepackage{amsmath}
\usepackage{slashed}
\usepackage{graphicx}

\usepackage[margin=1in,includeheadfoot]{geometry}
\usepackage{amsmath}
\usepackage{slashed}
\usepackage{graphicx}
\usepackage{mathrsfs}
\usepackage{txfonts}
\usepackage{amssymb,dsfont}
\usepackage{enumerate}
\usepackage{slashed}
\usepackage{fancyhdr}
\usepackage{aliascnt}
\usepackage{pdfsync}
\usepackage{hyperref}

\newtheorem{thm}{Theorem}[section]
\newtheorem{cor}[thm]{Corollary}
\newtheorem{lem}[thm]{Lemma}
\newtheorem{prop}[thm]{Proposition}
\theoremstyle{definition}

\theoremstyle{remark}
\newtheorem{rem}[thm]{Remark}
\numberwithin{equation}{section}

\newcommand{\dis}{\displaystyle}

\newcounter{stepnum}

\def\bee{\begin{eqnarray}}
\def\beee{\begin{eqnarray*}}
\def\eee{\end{eqnarray}}
\def\eeee{\end{eqnarray*}}

\def\ba{\begin{array}}
\def\ea{\end{array}}

\def\R{\mathbb R}

\begin{document}

\title[Liouville equation]{The $C^0$-convergence at the Neumann boundary for Liouville equations }

\author[Bi]{Yuchen Bi}
\address{Institute of Mathematics, Academy of Mathematics and Systems Science, University of Chinese Academy of Sciences,
Beijing, 100190, P. R. China}
\email{biyuchen15@mails.ucas.ac.cn}

\author[Li]{Jiayu Li}%
\address{School of Mathematical Sciences, University of Science and Technology of China, Hefei, 230026, P. R. China}%
\email{jiayuli@ustc.edu.cn}%

\author[Liu]{Lei Liu}
\address{School of Mathematics and Statistics \& Hubei Key Laboratory of
Mathematical Sciences, Central China Normal University, Wuhan 430079,
P. R. China}%
\email{leiliu2020@ccnu.edu.cn}

\author[Peng]{Shuangjie Peng}%
\address{School of Mathematics and Statistics \& Hubei Key Laboratory of
Mathematical Sciences, Central China Normal University, Wuhan 430079,
P. R. China}%
\email{sjpeng@ccnu.edu.cn}%

\thanks{}

\thanks{Lei Liu was supported in part by National Natural Science Foundation of China (Grant No. 12101255).}
\subjclass[2010]{}

\date{\today}
\begin{abstract}
In this paper, we study the blow-up analysis for a sequence of solutions to the Liouville type equation with exponential Neumann boundary condition. For interior case, i.e. the blow-up point is an interior point, Li \cite{Li} gave a uniform asymptotic estimate. Later, Zhang \cite{Zhang} and Gluck \cite{Gluck} improved Li's estimate in the sense of $C^0$-convergence by using the method of moving planes or classification of solutions of the linearized version of Liouville equation. If the sequence blows up at a boundary point, Bao-Wang-Zhou \cite{Bao-Wang-Zhou} proved a similar asymptotic estimate of Li \cite{Li}. In this paper, we will prove a $C^0$-convergence result in this boundary blow-up process. Our method is different from \cite{Zhang,Gluck} .
\end{abstract}
\maketitle
\section{introduction}

The compactness of a sequence of solutions to a nonlinear equation plays an important role in the study of the existence problem. For most of interesting geometric partial differential equations, one can not  use variational methods directly to get the existence result, such as harmonic maps, minimal surface, Liouville equation and so on. The main reason lies in the lack of compactness of the solution space. To overcome this type  of problem, an important method called blow-up analysis was employed to provide  better information of solution space.

The blow-up analysis for a sequence of solutions to Liouville type equation has been widely studied since the work of Breziz-Merle \cite{Brezis} where a concentration-compactness phenomenon of solutions was revealed. Later, Li-Shafrir \cite{Li-Shafrir} initiated to study the blow-up value at the blow-up point, which is quantized, i.e. at each blow-up point, the blow-up value is $8\pi m$ for some positive integer $m$. We emphasize  that there is no boundary condition in   Li-Shafrir \cite{Li-Shafrir}. If certain  boundary assumption is imposed, for example the oscillation on the boundary is uniformly bounded, then Li \cite{Li} proved that there is only one bubble at each blow-up point, which implies that  the blow-up value is $8\pi$. Roughly speaking, let $u_k$ be a sequence solutions to Liouville equation
\begin{equation}\label{equ:13}
-\Delta u=e^{u},\ \ in\ \  B_1(0)
\end{equation} with uniformly bounded energy
\begin{equation}\label{equ:14}
\int_{B_1(0)}e^{u_k(x)}dx\leq C
 \end{equation}and boundary condition
\begin{equation}\label{equ:15}
osc_{\partial B_1(0)}u_k\leq C.
 \end{equation}
Suppose $x=0$ is the only blow-up point for this sequence. Then Li \cite{Li} proved that there exists a sequence of points $x_k\to 0$ as $k\to\infty$, such that passing to a subsequence, there holds
\begin{equation}\label{ineq:04}
\left\|u_k(x)-u_k(x_k)-v\Big(\frac{x-x_k}{e^{-\frac{1}{2}u_k(x_k)}}\Big)\right\|_{C^0(B_1(0))}\leq C,\ \ \forall \ k,
\end{equation}
where $v(x)=-2\log(1+\frac{1}{8}|x|^2)$. Furthermore, if we assume a stronger boundary condition that $$osc_{\partial B_1(0)}u_k=o(1),$$ Chen-Lin \cite{Chen-Lin}, Zhang \cite{Zhang} and Gluck \cite{Gluck} established a type of  $C^0$-convergence result of \eqref{ineq:04} as follows
\begin{equation}\label{equ:00}
\lim_{k\to\infty}\left\|u_k(x)-u_k(x_k)-v\Big(\frac{x-x_k}{e^{-\frac{1}{2}u_k(x_k)}}\Big)\right\|_{C^0(B_1(0))}=0.
\end{equation}

\

Motivated by the question of prescribed geodesic curvature, it is also interesting to  study the Liouville equation with a Neumann boundary condition. For the blow-up analysis of Liouville type equation near the Neumann boundary, Guo-Liu \cite{Guo-Liu} proved Brezis-Merle type concentration phenomenon and Li-Shafrir type quantization property. Later, Bao-Wang-Zhou \cite{Bao-Wang-Zhou} extended  Li's work \cite{Li} to the boundary case. More precisely, they proved that there is only one bubble near a boundary blow-up point and a similar convergence result of \eqref{ineq:04} also holds near the boundary if  the oscillation on the boundary is uniformly bounded.

\

In this paper, we want to establish a similar $C^0$-convergence of \eqref{equ:00} near a boundary blow-up point. Before stating the main result, we make some notations.

Let $B_r(x_0)$ be the ball in $\R^2$ with radius $r$ centered at  $x_0$. Let $\partial B_r(x_0)$ be the boundary of $B_r(x_0)$. Denote $$B^+_r(x_0):=\{y=(y^1,y^2)\in B_r(x_0)\ |\ y^2> 0\}$$ and $$\partial^+B^+_r(x_0):=\{y=(y^1,y^2)\in \partial B_r(x_0)\ |\ y^2> 0\},\ \ \partial^0B^+_r(x_0):=\{y=(y^1,y^2)\in  B_r(x_0)\ |\ y^2= 0\}.$$ Denote $$\R^2_a:=\{y=(y^1,y^2)\in \R^2\ |\ y^2> -a\}$$ for some $a\geq 0$ and $$\partial \R^2_a:=\{y=(y^1,y^2)\in \R^2\ |\ y^2= -a\}.$$  For simplicity of notations, we always denote $B_1(0)$, $B^+_1(0)$, $B_R(0)$, $B^+_R(0)$ and $\R^2_0$ by $B$, $B^+$, $B_R$, $B^+_R$ and $\R^2_+$ respectively.

\

Now we consider the following equation
\begin{align}\label{equ:01}
\begin{cases}
  -\Delta u=e^{u},\ \ &in\ \ B^+,\\
  \dis\frac{\partial u}{\partial \vec{n}}=e^{\frac{u}{2}},\ \ &on\ \ \partial^0 B^+,
  \end{cases}
  \end{align} where $\vec{n}$ is the unit outer  normal vector on the boundary.

Let $u_k$ be a sequence of solutions of \eqref{equ:01} with uniformly bounded energy
\begin{equation}\label{ineq:01}
\int_{B}e^{u_k}dx+ \int_{\partial^0 B^+}e^{\frac{u_k}{2}}ds_x\leq C<\infty
\end{equation} and  $0$ be its only blow-up point in $B^+_1(0)$, i.e.
\begin{align}\label{ineq:02}
\max_{K\subset\subset\overline{B^+}\setminus \{0\}}u_k\leq C(K),\ \max_{\overline{B^+} }u_k\to +\infty.
\end{align}
Assume the  boundary condition as
\begin{equation}\label{ineq:03}
osc_{\partial^+ B^+}u_k=o(1).\end{equation}

Our main result is as follows
\begin{thm}\label{thm:02}
Let $u_k$ be a sequence of solutions of \eqref{equ:01} with conditions \eqref{ineq:01}, \eqref{ineq:02} and \eqref{ineq:03}. Then there exists a sequence of points $\{x_k\}\subset \overline{B^+}$ such that, passing to a subsequence, there hold:
\begin{itemize}
\item[(1)] $x_k\to 0$ and $u_k(x_k)=\sup_{B^+}u_k(x)\to +\infty$ as $k\to\infty$;

\

\item[(2)] Denote $\lambda_k=e^{-\frac{1}{2}u_k(x_k)}$ and $\Omega_k:=\{x\in\R^2\ |\ x_k+\lambda_kx\in B^+\}$. Then $$\lim_{k\to\infty}\frac{dist(x_k,\partial^0B^+)}{\lambda_k}=a<+\infty$$ and $\Omega_k\to \R^2_a$ as $k\to\infty$.

\

\item[(3)] Denote $v_k(x)=u_k(x_k+\lambda_kx)+2\log\lambda_k$. Then $$\lim_{k\to\infty}\left\| v_k(x)-v(x) \right\|_{C^1(B_R(0)\cap\Omega_k)}=0,\ \ \forall R>0,$$ where $$v(x)=\log\frac{8\lambda^2}{\Big[\lambda^2+(x^1+s_0)^2+\Big(x^2+a+\frac{\lambda}{\sqrt{2}}\Big)^2\Big]^2},\ \ x=(x^1,x^2)\in\R^2$$ for some $\lambda>0$ and $s_0\in\R$, which satisfies $v(0)=0$ and
\begin{align*}
\begin{cases}
-\Delta v=e^v\ \ &in\  \ \R^2_a,\\
\frac{\partial v}{\partial \overrightarrow{n}}=e^{\frac{v}{2}},\ \ &on\ \ \partial\R^2_a.
\end{cases}
\end{align*}

\

\item[(4)] The $C^0$-convergence: $$\lim_{k\to\infty}\left\| u_k(x)-u_k(x_k)-v\Big(\frac{x-x_k}{\lambda_k} \Big) \right\|_{C^0(B^+)}=0.$$
\end{itemize}

\end{thm}

\

To prove  Theorem \ref{thm:02}, we focus on establishing the $C^0$-convergence in the blow-up process, i.e. the fourth conclusion of Theorem \ref{thm:02} since  conclusions $(1)-(3)$ are standard now by \cite{Guo-Liu, Bao-Wang-Zhou,Zhang-Zhou}.

For conclusion $(4)$, by a standard blow-up analysis, it is  not hard to see that we only need to show
$$\lim_{k\to\infty}\left\| u_k(x)-u_k(x_k)-v\Big(\frac{x-x_k}{\lambda_k} \Big) \right\|_{osc(B^+\setminus B^+_{\lambda_kR_k}(x_k))}=0,
$$
for some $R_k\to+\infty$. We divide the oscillation estimate into two parts: $B^+_{\frac{1}{2}}(x_k)\setminus B^+_{\lambda_kR_k}(x_k)$ and $B^+\setminus B^+_{\frac{1}{2}}(x_k)$. For the first part, we will prove that both the tangential and radial oscillation are small. To control the tangential oscillation, we will use  Green's formula to derive a pointwise estimate of $\nabla v_k$, where there appears a bad term
$$\frac{|\log dist(x,\partial\R^2_+)|}{|x|^{3/2}}.
$$
 One may see that this term is not bounded if $x$ goes to the boundary even if $|x|$ is big. To overcome this obstacle, we will use the method of integration and the symmetry of Green's function for the unit ball, see Corollary \ref{cor:01}. To control the radial oscillation, the key point is to derive the following type energy decay
 $$\alpha(t)=4+\frac{O(1)}{t^p},\ \ t\in \Big[R_k,\frac{1}{2}\lambda_k^{-1}\Big]
 $$
  for some $p>0$, where $\alpha(t)$ is the energy defined by \eqref{equ:def-2}. This will be derived by combining  Pohozaev's type identity, the method of integration and above point estimate of $\nabla v_k$, see Lemma \ref{lem:blowup-value}. With the help of energy decay estimate, we will use an ODE method to get  the control of radial oscillation, see Proposition \ref{prop:01}. To estimate the second part, we will use the PDE's theory including $L^p$-theory, Green's formula and maximal principle to prove the oscillation is small.

\

The rest of paper is organized as follows. In Sect. \ref{sec:basic-lem}, we will prove some basic lemmas including a bounded oscillation lemma and a fast decay lemma near the boundary which will be used in our later proof. In Sect. \ref{sec:proof-thm}, we first derive a  point estimate of $\nabla v_k$. Then additionally  by using Pohozaev's type identity and the method of integration, we will establish a energy decay lemma which is crucial in our proof. Finally, we will prove Theorem \ref{thm:02} at the end of this section.

\

\section{Some basic lemmas}\label{sec:basic-lem}

\

In this section, we will  recall some classical blow-up analysis for Liouville type equation, and prove some lemmas which will be used in our later proof, such as a bounded oscillation estimate near the boundary, a fast decay lemma near the boundary and so on.

\

We start by recalling some standard theory in the blow-up analysis of Liouville type equation. See \cite{Brezis,Li-Shafrir,Li,Bao-Wang-Zhou,Guo-Liu}.   Let $x_k\in \overline{B^+}$ be the point such that $$u_k(x_k)=\sup_{B^+}u_k(x).$$ By \eqref{ineq:02}, it is easy to check that $$x_k\to 0\ \ and \ \ u_k(x_k)\to +\infty\ \ as \ \ k\to\infty.$$

Set $\lambda_k:=e^{ -\frac{1}{2}u_k(x_k)}$ and $$v_k(x):=u_k(x_k+\lambda_kx)+2\log\lambda_k.$$ Denote $$\Omega_k:=\{x\in\R^2 |x_k+\lambda_kx\in B^+\}$$ and $$d_k:=dist(x_k,\partial^0B^+).$$
Then we have the following two cases.

\

Case 1: $\lim_{k\to\infty}\dis\frac{d_k}{\lambda_k}=\infty$.

\

In this case, it is easy to check  that $v_k(x)$ is well defined  in the domain $\Omega_k$  which converges to the whole plane $\R^2$. By the standard theory of Liouville equation,  it is well known that $v_k(x)\to v(x)$ in $C^2_{loc}(\R^2)$, where $v(x)$ satisfies Liouville equation
\begin{equation}\label{equ:Liouville-inte-equ}
-\Delta v(x)=e^{v},\  \ in\ \ \R^2\  \ with\ \ \int_{\R^2}e^{v}dx\leq C<+\infty.
 \end{equation}
It follows from  the classification result  in \cite{Chen-Li} that  $$v(x)=-2\log\Big(1+\frac{|x|^2}{8}\Big),\ \ \int_{\R^2}e^{v}dx=8\pi.$$

%
%

\

Case 2: $\lim_{k\to\infty}\dis\frac{d_k}{\lambda_k}=a<\infty$.

\

In this case, the domain $\Omega_k$  converges to $\R^2_a:=\{x=(x^1,x^2)\in\R^2 |\ x^2\geq -a\}$.  Since $$v_k(0)=0,\,\,v_k(x)\leq 0,\ \ \forall\, x\in\Omega_k$$ and $$|\Delta v_k(x)|\leq C,\ \  \Big|\frac{\partial v_k}{\partial \overrightarrow{n}}(x)\Big|\leq C,$$ by a standard elliptic theory involving a Harnack inequality near a Neumann-type boundary (see Lemma A.2 in \cite{Jost-Wang-Zhou-Zhu}), we claim that
\begin{align}\label{ineq:05}
\|v_k\|_{L^\infty(B^+_R)}\leq C(R)\ \  uniformly \ \ in \ \ B^+_R.
 \end{align}
In fact, denoting $\tilde{x}_k=(0,-\frac{d_k}{\lambda_k})$ and setting $$\tilde{v}_k(x):=v_k(x+\tilde{x}_k),$$ we can check  that $\tilde{v}_k(x)\leq 0,\ \ \tilde{v}_k(-\tilde{x}_k)=0$ and
\begin{align*}
\begin{cases}
-\Delta \tilde{v}_k=e^{ \tilde{v}_k},\ \ &in\ \ B^+_{\frac{1}{2}\lambda_k^{-1}},\\
\dis\frac{\partial \tilde{v}_k}{\partial \overrightarrow{n}}=e^{\frac{1}{2}\tilde{v}_k},\ \ &on\ \  \partial^0B^+_{\frac{1}{2}\lambda_k^{-1}}.
\end{cases}
\end{align*}

 Let $w_k$ be the solution of
\begin{align*}
\begin{cases}
-\Delta w_k=-\Delta \tilde{v}_k,\ \ &in\ \ B^+_{2R},\\
\dis\frac{\partial w_k}{\partial \overrightarrow{n}}=0,\ \ &on\ \  \partial^0B^+_{2R},\\
w_k=0,\ \ &on\ \  \partial^+B^+_{2R}.
\end{cases}
\end{align*}
Extend $w_k$ evenly, then $\|w_k\|_{L^\infty(B^+_{2R})}\leq C(R)$. Set $\eta_k=\tilde{v}_k-w_k$, then
\begin{align*}
\begin{cases}
-\Delta \eta_k=0,\ \ &in\ \ B^+_{2R},\\
\dis\frac{\partial \eta_k}{\partial \overrightarrow{n}}=\frac{\partial \tilde{v}_k}{\partial \overrightarrow{n}},\ \ &on\ \  \partial^0B^+_{2R},\\
\eta_k=\tilde{v}_k,\ \ &on\ \  \partial^+B^+_{2R}.
\end{cases}
\end{align*} By the Harnack inequality of Lemma A.2 in \cite{Jost-Wang-Zhou-Zhu}) and the fact that $|\eta_k(-\tilde{x}_k)|\leq C<+\infty$, we have $\|\eta_k\|_{L^\infty(B^+_{\frac{3}{2}R})}\leq C(R)$, which implies the claim \eqref{ineq:05} immediately.

By \eqref{ineq:05} and the standard elliptic theory, we know that $v_k$ converges in $C^2(B_R(0)\cap \Omega_k\cap\R^2_a)$ to a function $v$ which satisfies
\begin{align*}
\begin{cases}
-\Delta v=e^v\ \ &in\  \ \R^2_a,\\
\dis\frac{\partial v}{\partial \overrightarrow{n}}=e^{\frac{v}{2}},\ \ &on\ \ \partial\R^2_a,
\end{cases}
\end{align*} and $$\int_{\R^2_a}e^vdx+\int_{\partial\R^2_a}e^{\frac{v}{2}}ds_x\leq C<\infty.$$

Using the classification result in \cite{Li-Zhu}, we know $$v(x)=\log\frac{8\lambda^2}{\Big[\lambda^2+(x^1+s_0)^2+\Big(x^2+a+\frac{\lambda}{\sqrt{2}}\Big)^2\Big]^2},\ \ x=(x^1,x^2)\in\R^2$$ for some $\lambda>0$ and $s_0\in\R$ and $$\int_{\R^2_a}e^vdx+\int_{\partial\R^2_a}e^{\frac{v}{2}}ds_x=4\pi.$$

Moreover, we can choose a sequence $R_k\to\infty$ such that $\lambda_kR_k\to 0$ and passing to a subsequence, there holds
\begin{align}\label{equ:11}
\|v_k(x)-v(x)\|_{C^1(B_{2R_k}\cap \Omega_k\cap\R^2_a)}=o(1)\ \ and\ \ \int_{ B_{R_k}\cap \Omega_k}e^{v_k(x)}dx+ \int_{ B_{R_k}\cap \partial^0 \Omega_k}e^{\frac{1}{2}v_k(x)}ds_x =4\pi+o(1),
\end{align} where $\partial^0\Omega_k:=\{x=(x^1,x^2)\in\overline{\Omega}_k\ \ |\ \ x^2=-\frac{d_k}{\lambda_k}\}.$

\

\begin{lem}\label{lem:01}
Under the assumptions of Theorem \ref{thm:02},  case $(1)$ will not happen, i.e. only case $(2)$ holds. Moreover, we have $$\int_{\Omega_k\setminus B_{R_k}}e^{v_k(x)}dx+ \int_{ \partial^0 \Omega_k \setminus B_{R_k} }e^{\frac{1}{2}v_k(x)}ds_x =o(1).$$
\end{lem}
\begin{proof}
It follows from  the results in \cite{Bao-Wang-Zhou} or \cite{Zhang-Zhou}  that
\begin{equation}\label{equ:16}
\int_{B^+}e^{u_k(x)}dx+\int_{\partial^0B^+}e^{\frac{u_k(x)}{2}}ds_x=4\pi+o(1).
\end{equation}
If  case $(1)$ holds, then it is easy to get that $$\int_{B^+}e^{u_k(x)}dx\geq 8\pi,$$ which is a contradiction. Thus  case $(1)$ will not happen. Since  case $(2)$ implies \eqref{equ:11},  the second conclusion of the lemma follows immediately from  \eqref{equ:16}.
\end{proof}

\

\begin{rem}\label{rem}
So far, one can see that we have proved conclusions $(1)-(3)$ of Theorem \ref{thm:02}. In the sequel,  we focus on the proof of conclusion $(4)$.
\end{rem}

\

Since the solution has  only one bubble, by the selection of bubbling areas in \cite{Li-Shafrir}, we have the following lemma.
\begin{lem}\label{lem:04}
Under the assumptions of Theorem \ref{thm:02}, we have $$u_k(x)+2\log|x-x_k|\leq C.$$
\end{lem}
\begin{proof}
From \eqref{equ:11}, we see $$v_k(x)+2\log|x|\leq C,\ \ |x|\leq 2R_k,\ x\in\Omega_k,$$ which implies $$u_k(x)+2\log|x-x_k|\leq C,\ \ |x-x_k|\leq 2\lambda_k R_k,\ x\in B^+.$$

Now we claim $$u_k(x)+2\log|x-x_k|\leq C,\ \  x\in B^+.$$ Otherwise, there exists $y_k\in \overline{B^+}$ such that $|y_k-x_k|\geq 2\lambda_kR_k$ and $$u_k(y_k)+2\log|y_k-x_k|\to +\infty.$$ Denote $t_k:=\frac{1}{2}|x_k-y_k|$ and $$\phi_k(x):=u_k(x)+2\log(t_k-|x-y_k|).$$ Let $p_k$ be the maximal point such that $$\phi_k(p_k)=\max_{x\in \overline{B^+_{t_k}(y_k)}}\phi_k(x).$$ It is easy to check that $$\phi_k(p_k)\geq \phi_k(y_k)\to+\infty,$$ which also implies $u_k(p_k)\to +\infty$.

Denote $$s_k:=\frac{1}{2}(t_k-|p_k-y_k|)\ \ and \ \ \epsilon_k:=e^{-\frac{1}{2}u_k(p_k)}.$$ Since $\phi_k(p_k)\to +\infty$, we have $$\frac{s_k}{\epsilon_k}\to +\infty.$$ Noting that for any $x\in B^+_{s_k}(p_k)$, there hold $$u_k(x)+2\log|t_k-|x-y_k||\leq \phi_k(p_k)=u_k(p_k)+2\log(2s_k)$$ and $$t_k-|x-y_k|\geq t_k-|x-p_k|-|p_k-y_k|\geq s_k,$$ we get $$u_k(x)\leq u_k(p_k)+2\log 2,\ \ \forall \ x\in B^+_{s_k}(p_k).$$

Set $$w_k(x)=u_k(p_k+\epsilon_kx)+2\log\epsilon_k, p_k+\epsilon_kx\in B^+.$$ It is easy to check that $$ p_k+\epsilon_kx\in B^+, w_k(0)=0,\,w_k(x)\leq 2\log2,\ \ \forall \,\ |x|\leq \frac{s_k}{\epsilon_k}.$$

Now, we distinguish the following two cases.

\

Case 1: $\lim_{k\to\infty}\dis\frac{dist(p_k,\partial^0B^+)}{\epsilon_k}=+\infty$.

\

In this case, we see that $w_k(x)$ is well defined  in a domain which converges to the whole plane $\R^2$ as $k\to \infty$. Then by the standard theory of Liouville type equation, we know that $w_k(x)\to w(x)$ in $C^2_{loc}(\R^2)$, where $w(x)$ satisfies Liouville equation \eqref{equ:Liouville-inte-equ}. Moreover, we can select a sequence $R^1_k\to\infty$ such that $R^1_k\epsilon_k=o(1)s_k$ and passing to a subsequence, there holds $$\|w_k(x)-w(x)\|_{C^0(B_{R^1_k})}\to 0.$$ It is easy to check that $$B^+_{\lambda_kR_k}(x_k)\cap B_{\epsilon_k R^1_k}(p_k)=\emptyset.$$
 Hence $$\int_{B^+}e^{u_k(x)}dx\geq \int_{B_{\epsilon_kR^1_k}}e^{u_k(x)}dx=8\pi+o(1),$$ which is a contradiction.

\

Case 2: $\lim_{k\to\infty}\dis\frac{dist(p_k,\partial^0B^+)}{\epsilon_k}=b<+\infty$.

\

In this case, the domain $\Omega^1_k=\{x\in\R^2\ |\ p_k+\epsilon_kx\in B^+\}$ converges to  $\R_b^2$ as $k\to \infty$. By a similar argument as before, we know that $w_k(x)\to w(x)$ in $C^1(B_R(0)\cap \Omega^1_k)$, where $w(x)$ is a solution of
\begin{align*}
\begin{cases}
-\Delta w=e^w\ \ &in\  \ \R^2_b,\\
\dis\frac{\partial w}{\partial \overrightarrow{n}}=e^{\frac{w}{2}},\ \ &on\ \ \partial\R^2_b.
\end{cases}
\end{align*}
Moreover, we can select a sequence $R^1_k\to\infty$ such that $R^1_k\epsilon_k=o(1)s_k$ and passing to a subsequence, there holds $$\|w_k(x)-w(x)\|_{C^1(B_{R^1_k}\cap \Omega^1_k)}\to 0.$$ It is easy to check that $$B^+_{\lambda_kR_k}(x_k)\cap B^+_{\epsilon_k R^1_k}(p_k)=\emptyset.$$
As a result,
\begin{align*}
&\int_{B^+}e^{u_k(x)}dx+\int_{\partial^0B^+}e^{\frac{1}{2}u_k(x)}ds_x\\ \geq & \int_{B^+_{\lambda_kR_k}(x_k)}e^{u_k(x)}dx+ \int_{\partial^0B^+_{\lambda_kR_k}(x_k)}e^{\frac{1}{2}u_k(x)}ds_x+ \int_{B^+_{\epsilon_kR^1_k}(p_k)}e^{u_k(x)}dx+ \int_{\partial^0B^+_{\epsilon_kR^1_k}(p_k)}e^{\frac{1}{2}u_k(x)}ds_x\\=&8\pi+o(1),
\end{align*}
which is also a contradiction.

We proved the lemma.

\end{proof}

\

With the help of Lemma \ref{lem:04}, we have the following bounded oscillation estimate near the boundary.
\begin{lem}\label{lem:osc}
Let $u_k$ be a sequence of solutions of \eqref{equ:01} satisfying \eqref{ineq:01}-\eqref{ineq:02}. Then $$osc_{ B^+_{\frac{1}{2}d_k(x)}(x)}u_k\leq C,\ \ \forall x\in B^+\setminus \{x_k\},$$ where $d_k(x)=|x-x_k|$ and $C$ is a universal constant independent of $k$ and $x$.
\end{lem}
\begin{proof}
Let $w_k$ be the solution of
\begin{align*}
\begin{cases}
-\Delta w_k=-\Delta u_k,\ \ &in\ \ B^+,\\
\dis\frac{\partial w_k}{\partial \overrightarrow{n}}=0,\ \ &on\ \  \partial^0 B^+,\\
w_k=0,\ \ &on\ \  \partial^+ B^+.
\end{cases}
\end{align*}
Define
\begin{equation}
\phi_k(x):=-\frac{1}{\pi}\int_{\partial^0 B^+}\log|x-y|e^{\frac{u_k(y)}{2}}ds_y,\ \ x\in B^+.
\end{equation}
 Then we can check
\begin{align}
\begin{cases}
-\Delta \phi_k=0,\ \ &in\ \ B^+,\\
\dis\frac{\partial \phi_k}{\partial \overrightarrow{n}}=e^{\frac{u_k(x)}{2}},\ \ &on\ \  \partial^0 B^+.
\end{cases}
\end{align}
Letting $\eta_k=u_k-w_k-\phi_k$, there holds
\begin{align}
\begin{cases}
-\Delta \eta_k=0,\ \ &in\ \ B^+,\\
\dis\frac{\partial \eta_k}{\partial \overrightarrow{n}}=0,\ \ &on\ \  \partial^0 B^+,\\
\eta_k=u_k,\ \ &on\ \  \partial^+ B^+.
\end{cases}
\end{align} Extending $\eta_k$ evenly, by maximal principle, we have $$osc_{B^+} \eta_k(x)\leq osc_{\partial^+B^+} u_k(x)=o(1).$$

Extending $w_k$ and $u_k$ evenly (we still use the same notations), by Green's formula, we have $$w_k(x)=\int_{B_1(0)}G(x,y)e^{u_k(y)}dy,$$ where $$G(x,y)=-\frac{1}{2\pi}\log |x-y|+H(x,y)$$ is the Green's function on $B_1$ with respect to the Dirichlet boundary and $H(x,y)$ is a smooth harmonic function.

For any $x\in B^+\setminus \{x_k\}$ and $p_1,p_2\in B^+_{\frac{1}{2}d_k(x)}(x)$, we have
\begin{align*}
&w_k(p_1)-w_k(p_2)+\phi_k(p_1)-\phi_k(p_2)\\
=&\int_{B_1(0)}\left(G(p_1,y)-G(p_2,y)\right) e^{u_k(y)}dy-\frac{1}{\pi}\int_{\partial^0 B^+}\log\frac{|p_1-y|}{|p_2-y|}e^{\frac{u_k(y)}{2}}ds_y\\
=&-\frac{1}{2\pi}\int_{ B}\log\frac{|p_1-y|}{|p_2-y|}e^{u_k(y)}dy-\frac{1}{\pi}\int_{\partial^0 B^+}\log\frac{|p_1-y|}{|p_2-y|}e^{\frac{u_k(y)}{2}}ds_y+O(1)\\
=&-\frac{1}{2\pi}\int_{ B_{\frac{3}{4}d_k(x)}(x)}\log\frac{|p_1-y|}{|p_2-y|}e^{u_k(y)}dy-\frac{1}{\pi}\int_{\partial^0 B_{\frac{3}{4}d_k(x)}^+(x)}\log\frac{|p_1-y|}{|p_2-y|}e^{\frac{u_k(y)}{2}}ds_y\\
& -\frac{1}{2\pi}\int_{ B\setminus B_{\frac{3}{4}d_k(x)}(x)}\log\frac{|p_1-y|}{|p_2-y|}e^{u_k(y)}dy-\frac{1}{\pi}\int_{\partial^0 (B^+\setminus B_{\frac{3}{4}d_k(x)}^+(x))}\log\frac{|p_1-y|}{|p_2-y|}e^{\frac{u_k(y)}{2}}ds_y+O(1)\\
=&\mathbf{I}+\mathbf{II}+O(1).
\end{align*}

We first estimate $\mathbf{II}$. Noting that $|y-x|\geq\frac{3}{4}d_k(x)$, we find
$$\left|\log\frac{|p_1-y|}{|p_2-y|}\right|\leq C,
$$
 which implies $\mathbf{II}=O(1)$.

For $\mathbf{I}$, a direct computation yields
\begin{align*}
|\mathbf{I}|&\leq C\int_{ B_{\frac{3}{4}d_k}(x)}\left|\log\frac{|p_1-y|}{|p_2-y|}\right|e^{u_k(y)}dy+C\int_{\partial^0 B_{\frac{3}{4}d_k}^+(x)}\left|\log\frac{|p_1-y|}{|p_2-y|}\right|e^{\frac{u_k(y)}{2}}ds_y\\
&= C\int_{ B_{\frac{3}{4}}(0)}\left|\log\frac{|p_1-x-d_k(x)y|}{|p_2-x-d_k(x)y|}\right|e^{u_k(x+d_k(x)y)}d^2_k(x)dy\\&\quad+C\int_{\partial^0 B_{\frac{3}{4}}^+(0)}\left|\log\frac{|p_1-x-d_k(x)y|}{|p_2-x-d_k(x)y|}\right|e^{\frac{u_k(x+d_k(x)y)}{2}}d_k(x)ds_y\\
&= C\int_{ B_{\frac{3}{4}}(0)}\left|\log\frac{|\frac{p_1-x}{d_k(x)}-y|}{|\frac{p_2-x}{d_k(x)}-y|}\right|e^{u_k(x+d_k(x)y)}d^2_k(x)dy+C\int_{\partial^0 B_{\frac{3}{4}}^+(0)}\left|\log\frac{|\frac{p_1-x}{d_k(x)}-y|}{|\frac{p_2-x}{d_k(x)}-y|}\right|e^{\frac{u_k(x+d_k(x)y)}{2}}d_k(x)ds_y.
\end{align*}
 Noting
 $$|x+d_k(x)y-x_k|\geq |x-x_k|-|d_k(x)y|\geq \frac{1}{4}d_k(x),\ \ \forall y\in B_{\frac{3}{4}}(0),
 $$
  by Lemma \ref{lem:04}, we have $$u_k(x+d_k(x)y)\leq C-2\log |x+d_k(x)y-x_k|\leq C-2\log d_k(x),\ \ \forall y\in B_{\frac{3}{4}}(0).$$
   From
   $$\left|\frac{p_1-x}{d_k(x)}\right|\leq \frac{1}{2},\ \ \left|\frac{p_2-x}{d_k(x)}\right|\leq \frac{1}{2},$$
   one can easily find that $\mathbf{I}=O(1)$, which implies the conclusion of the lemma.
\end{proof}

\

At the end of this section, we prove the following fast decay lemma, which will be used in next section to estimate a upper bound of solutions and to control the radial oscillation of solutions. This type  of fast decay definition was firstly introduced in  \cite{Lin-Wei-Zhang}.
\begin{lem}\label{lem:fast-decay}
There exists a number sequence $\{N_k\}$ which tends to $+\infty$ as $k\to\infty$, such that $$v_k(x)+2\log|x|\leq -N_k,\ \ |x|\geq R_k,\ \ x\in\Omega_k.$$
\end{lem}
\begin{proof}
By \eqref{equ:11}, we know $$v_k(x)+2\log|x|\leq -2\log R_k+C,\ \ \forall \ R_k\leq |x|\leq 2R_k,\ x\in\Omega_k.$$ Denote $\tilde{x}_k:=(0,-\frac{d_k}{\lambda_k})$  and
\begin{equation}\label{equ:def-1}
\tilde{v}_k(x)=v_k(x+\tilde{x}_k),\ \ x\in B^+_{\frac{3}{4}\lambda_k^{-1}}.\end{equation} Noting that $|\tilde{x}_k|\leq C$,  it is easy to check that
\begin{equation}\label{ineq:07}
\tilde{v}_k(x)+2\log|x|\leq -2\log R_k+C,\ \ \forall x\in B_{2R_k}^+\setminus B^+_{R_k}.
 \end{equation}
Set
\begin{equation}\label{equ:def-3}
\tilde{v}^*_k(r):=\frac{1}{\pi r}\int_{\partial^+B^+_{r}}\tilde{v}_k(x)d\theta,\ \ 0<r\leq\frac{3}{4}\lambda_k^{-1}.
\end{equation}

On one hand, by Lemma \ref{lem:osc}, we have
\begin{equation}\label{ineq:06}
osc_{\partial^+ B^+_{r}} \tilde{v}_k(x)\leq osc_{\frac{1}{2}r\leq |x|\leq \frac{3}{2}r,\ x\in\Omega_k}v_k(x)\leq C,\ \ \forall \ r\in \Big[R_k,\frac{1}{4}\lambda_k^{-1}\Big].
\end{equation}
Combining  \eqref{ineq:07} and \eqref{ineq:06}, we obtain
\begin{equation}\label{ineq:08}
\tilde{v}^*_k(2R_k)+2\log(2R_k)\leq -2\log R_k+C.
 \end{equation}

On the other hand, defining
\begin{equation}\label{equ:def-2}
\alpha(r):=\frac{1}{\pi}\left(\int_{B^+_{r}}e^{\tilde{v}_k(x)}dx+ \int_{\partial^0B^+_{r}}\frac{\partial}{\partial\vec{n}}\tilde{v}_k(x)ds_x\right)
\end{equation} and making a direct computation, we have
\begin{align}\label{equ:02}
\frac{d}{dr}\tilde{v}^*_k(r)&=\frac{1}{\pi r}\int_{\partial^+B^+_{r}}\frac{\partial}{\partial\vec{n}}\tilde{v}_k(x)d\theta\notag\\
&=\frac{1}{\pi r}\int_{B^+_{r}}\Delta\tilde{v}_k(x)dx- \frac{1}{\pi r}\int_{\partial^0B^+_{r}}\frac{\partial}{\partial\vec{n}}\tilde{v}_k(x)ds_x\notag\\
&=-\frac{1}{\pi r}\int_{B^+_{r}}e^{\tilde{v}_k(x)}dx- \frac{1}{\pi r}\int_{\partial^0B^+_{r}}\frac{\partial}{\partial\vec{n}}\tilde{v}_k(x)ds_x:=-\frac{1}{ r}\alpha(r).
\end{align}

Thus, for any $t\in [2R_k,\frac{1}{4}\lambda_k^{-1}]$, there holds
\begin{align*}
\tilde{v}^*_k(t)+2\log t&=\tilde{v}^*_k(2R_k)+2\log (2R_k)+\int_{2R_k}^{t}\frac{d}{dr}\left(\tilde{v}^*_k(r)+2\log r\right)dr\\
&=\tilde{v}^*_k(2R_k)+2\log (2R_k)+\int_{2R_k}^{t}\frac{2-\alpha(r)}{r}dr\\
&\leq \tilde{v}^*_k(2R_k)+2\log (2R_k)\leq -2\log R_k+C,
\end{align*}
where the last second inequality follows from
 $$\alpha(r)\geq \frac{1}{\pi}\Big(\int_{B^+_{R_k}}e^{v_k(x)}dx+ \int_{B_{R_k}\cap\partial^0\Omega_k}e^{\frac{1}{2}v_k(x)}ds_x\Big)=\frac{4\pi+o(1)}{\pi}=4+o(1)$$
 and the last inequality follows from \eqref{ineq:08}.

Combining this with \eqref{ineq:06} and \eqref{ineq:07}, we get
\begin{equation*}
\tilde{v}_k(x)+2\log|x|\leq -2\log R_k+C,\ \ \forall x\in B_{\frac{1}{4}\lambda_k^{-1}}^+\setminus B^+_{R_k},
 \end{equation*}
 and hence
 \begin{equation*}
v_k(x)+2\log|x|\leq -2\log R_k+C,\ \ \forall x\in B_{\frac{1}{4}\lambda_k^{-1}}\setminus B_{R_k},\ x\in\Omega_k.
 \end{equation*}
Combining this with the fact that
 $$osc_{x\in B_{\lambda_k^{-1}}\setminus B_{\frac{1}{4}\lambda^{-1}_k},\ x\in\Omega_k}v_k\leq C,$$
and using Lemma \ref{lem:osc}, we conclude that  \begin{equation*}
v_k(x)+2\log|x|\leq -2\log R_k+C,\ \ \forall x\in B_{\lambda_k^{-1}}\setminus B_{R_k},\ x\in\Omega_k.
 \end{equation*}
 Therefore the conclusion of the lemma follows immediately by taking $N_k=2\log R_k+C$.
\end{proof}

\

\section{Proof of Theorem \ref{thm:02}}\label{sec:proof-thm}

\

In this section, we  first derive a uniformly upper bound of $v_k$ in its whole definition domain and then prove a pointwise estimate of $\nabla v_k$ when $|x|$ is large , which yields the control of tangential oscillation. We will also derive a energy decay lemma which is crucial in our proof. The proof of Theorem \ref{thm:02} will be given at the end of this section.

\begin{lem}\label{lem:02}
For any $\delta\in (0,\frac{1}{2})$, we have $$v_k(x)\leq -(4-2\delta)\log |x|+C,\ \  x\in\Omega_k,\ |x|\geq 2,$$ for  $k$  large enough .
\end{lem}
\begin{proof}
The proof of the lemma is more or less standard now. See for example \cite{Bao-Wang-Zhou} and \cite{Bartolucci-Chen-Lin-Tarantello}. For reader's convenience and our later proof, here we give a detailed proof.

\

By Lemma \ref{lem:osc}, we know $$osc_{\Omega_k\setminus B_{\frac{1}{2}\lambda_k^{-1}}}v_k(x)\leq C.$$ So it suffices  to prove that
$$v_k(x)\leq -(4-2\delta)\log |x|+C,\ \ 2\leq |x|\leq \frac{1}{2}\lambda_k^{-1},\ x\in\Omega_k.$$

By \eqref{equ:11}, we know $$v_k(x)\leq -4\log |x|+C,\ \ 2\leq |x|\leq R_k,\ x\in\Omega_k.$$  Hence,  we only need to show that $$v_k(x)\leq -(4-2\delta)\log |x|,\ \  R_k\leq |x|\leq \frac{1}{2}\lambda^{-1}_k,\ x\in\Omega_k.$$
Considering  $\lim_{k\to\infty}\frac{d_k}{\lambda_k}=a<\infty$, we find that the above estimate   is equivalent to
\begin{equation}\label{equ:12}
\tilde{v}_k(x)\leq -(4-2\delta)\log |x|,\ \  x\in B^+_{\frac{1}{2}\lambda^{-1}_k}\setminus B^+_{R_k},
\end{equation}
where $\tilde{v}_k(x)$ is defined by \eqref{equ:def-1} which satisfies
\begin{align*}
\begin{cases}
-\Delta \tilde{v}_k(x)=e^{\tilde{v}_k(x)},\ \ & in\ \ B^+_{\frac{1}{2}\lambda_k^{-1}},\\
\dis\frac{\partial \tilde{v}_k(x)}{\partial \overrightarrow{n}}=e^{\frac{\tilde{v}_k(x)}{2}},\ \ & on\ \ \partial^0 B^+_{\frac{1}{2}\lambda_k^{-1}},\\
\tilde{v}_k(x)\leq 0,\ \ & on\ \ \partial^+ B^+_{\frac{1}{2}\lambda_k^{-1}}.
\end{cases}
\end{align*}

Let $w_k$ be the solution of
\begin{align}
\begin{cases}\label{equ:04}
-\Delta w_k=-\Delta \tilde{v}_k,\ \ &in\ \ B^+_{\frac{1}{2}\lambda_k^{-1}},\\
\dis\frac{\partial w_k}{\partial \overrightarrow{n}}=0,\ \ &on\ \  \partial^0 B^+_{\frac{1}{2}\lambda_k^{-1}},\\
w_k=0,\ \ &on\ \  \partial^+ B^+_{\frac{1}{2}\lambda_k^{-1}}.
\end{cases}
\end{align}

Define
\begin{equation}\label{equ:10}
\phi_k(x):=-\frac{1}{\pi}\int_{\partial^0 B^+_{\frac{1}{2}\lambda_k^{-1}}}\log|x-y|e^{\frac{\tilde{v}_k(y)}{2}}ds_y,\ \ x\in B^+_{\frac{1}{2}\lambda_k^{-1}}.\end{equation}
 It is easy to see that
\begin{align}
\begin{cases}\label{equ:06}
-\Delta \phi_k=0,\ \ &in\ \ B^+_{\frac{1}{2}\lambda_k^{-1}},\\
\dis\frac{\partial \phi_k}{\partial \overrightarrow{n}}=e^{\frac{\tilde{v}_k(X)}{2}},\ \ &on\ \  \partial^0 B^+_{\frac{1}{2}\lambda_k^{-1}}.
\end{cases}
\end{align}

Letting $\eta_k=\tilde{v}_k-w_k-\phi_k$, there holds
\begin{align}
\begin{cases}\label{equ:05}
-\Delta \eta_k=0,\ \ &in\ \ B^+_{\frac{1}{2}\lambda_k^{-1}},\\
\dis\frac{\partial \eta_k}{\partial \overrightarrow{n}}=0,\ \ &on\ \  \partial^0 B^+_{\frac{1}{2}\lambda_k^{-1}},\\
\eta_k=\tilde{v}_k,\ \ &on\ \  \partial^+ B^+_{\frac{1}{2}\lambda_k^{-1}}.
\end{cases}
\end{align} Extending $\eta_k$ evenly, by maximal principle, we have
\begin{align}
\eta_k(x)-\eta_k(0)\leq osc_{ B^+_{\frac{1}{2}\lambda_k^{-1}}}\eta_k(x)\leq osc_{\partial^+ B^+_{\frac{1}{2}\lambda_k^{-1}}}\tilde{v}_k(x)\leq osc_{B^+\setminus B^+_{\frac{1}{4}}}u_k(x)\leq C,
\end{align} where the last inequality follows from Lemma \ref{lem:osc}.

Now  we  show that
\begin{equation}\label{ineq:09}
w_k(x)+\phi_k(x)-w_k(0)-\phi_k(0)\leq -(4-2\delta)\log |x|+O(1),\ \ R_k\leq |x|\leq \frac{1}{2}\lambda_k^{-1},
\end{equation}
 which implies \eqref{equ:12} since $\tilde{v}_k(0)\leq C$.

Extending $w_k$ evenly, by Green's formula, we have
\begin{align*}
w_k(x)=\int_{B_{\frac{1}{2}\lambda_k^{-1}}}G(2\lambda_k x,2\lambda_k y)e^{\tilde{v}_k(y)}dy,\ \ \forall \ x\in B_{\frac{1}{2}\lambda_k^{-1}}.
\end{align*}

A direct computation yields
\begin{align}\label{equ:07}
&w_k(x)+\phi_k(x)-w_k(0)-\phi_k(0)\notag\\
=&-\frac{1}{2\pi}\int_{B_{\frac{1}{2}\lambda_k^{-1}}}\log\frac{|x-y|}{|y|}e^{\tilde{v}_k(y)}dy -\frac{1}{\pi}\int_{\partial^0 B^+_{\frac{1}{2}\lambda_k^{-1}}}\log\frac{|x-y|}{|y|}e^{\frac{1}{2}\tilde{v}_k(y)}ds_y+O(1)\notag\\
=&-\log |x|\left[\frac{1}{2\pi}\int_{B_{\frac{1}{2}\lambda_k^{-1}}}e^{\tilde{v}_k(y)}dy +\frac{1}{\pi}\int_{\partial^0B^+_{\frac{1}{2}\lambda_k^{-1}}}e^{\frac{1}{2}\tilde{v}_k(y)}ds_y\right] \notag\\ &-\frac{1}{2\pi}\int_{B_{\frac{1}{2}\lambda_k^{-1}}}\log\frac{|x-y|}{|x||y|}e^{\tilde{v}_k(y)}dy -\frac{1}{\pi}\int_{\partial^0B^+_{\frac{1}{2}\lambda_k^{-1}}}\log\frac{|x-y|}{|x||y|}e^{\frac{1}{2}\tilde{v}_k(y)}ds_y+O(1)\notag\\ =&\alpha\Big(\frac{1}{2}\lambda_k^{-1}\Big) -\frac{1}{2\pi}\int_{B_{\frac{1}{2}\lambda_k^{-1}}}\log\frac{|x-y|}{|x||y|}e^{\tilde{v}_k(y)}dy -\frac{1}{\pi}\int_{\partial^0B^+_{\frac{1}{2}\lambda_k^{-1}}}\log\frac{|x-y|}{|x||y|}e^{\frac{1}{2}\tilde{v}_k(y)}ds_y+O(1),
\end{align} where $\alpha(r)$ is defined by \eqref{equ:def-2}.

We claim that
\begin{equation}\label{equ:08}
-\frac{1}{2\pi}\int_{B_{\frac{1}{2}\lambda_k^{-1}}}\log\frac{|x-y|}{|x||y|}e^{\tilde{v}_k(y)}dy=o(1)\log|x|,\ \ \forall\, R_k\leq |x|\leq\frac{1}{2}\lambda_k^{-1}
\end{equation} and
\begin{equation}\label{equ:09}
-\frac{1}{\pi}\int_{\partial^0B^+_{\frac{1}{2}\lambda_k^{-1}}}\log\frac{|x-y|}{|x||y|}e^{\frac{1}{2}\tilde{v}_k(y)}ds_y=o(1)\log|x|,\ \ \forall\, R_k\leq |x|\leq\frac{1}{2}\lambda_k^{-1}.
\end{equation}
In fact,
\begin{align*}
&\int_{B_{\frac{1}{2}\lambda_k^{-1}}}\log\frac{|x-y|}{|x||y|}e^{\tilde{v}_k(y)}dy\\
=& \int_{\{y\in B_{\frac{1}{2}\lambda_k^{-1}}\ |\ |y|\geq 2|x| \}}\log\frac{|x-y|}{|x||y|}e^{\tilde{v}_k(y)}dy +\int_{\{y\in B_{2|x|}\ |\ |y-x|\geq \frac{1}{2}|x|,\ |y|\geq R_k\}}\log\frac{|x-y|}{|x||y|}e^{\tilde{v}_k(y)}dy\\
&+\int_{\{ |y-x|\leq \frac{1}{2}|x|\}}\log\frac{|x-y|}{|x||y|}e^{\tilde{v}_k(y)}dy+\int_{\{ |y|\leq R_k\}}\log\frac{|x-y|}{|x||y|}e^{\tilde{v}_k(y)}dy\\
=& \mathbf{I}_1+\mathbf{I}_2+\mathbf{I}_3+\mathbf{I}_4.
\end{align*}
Similarly,
\begin{align*}
&\int_{\partial^0 B^+_{\frac{1}{2}\lambda_k^{-1}}}\log\frac{|x-y|}{|x||y|}e^{\frac{1}{2}\tilde{v}_k(y)}ds_y\\
=& \int_{\{y\in \partial^0 B^+_{\frac{1}{2}\lambda_k^{-1}}\ |\ |y|\geq 2|x| \}}\log\frac{|x-y|}{|x||y|}e^{\frac{1}{2}\tilde{v}_k(y)}ds_y +\int_{\{y\in \partial^0 B^+_{\frac{1}{2}\lambda_k^{-1}}\ |\ |y|\leq 2|x|,  \ |y-x|\geq \frac{1}{2}|x|,\ |y|\geq R_k\}}\log\frac{|x-y|}{|x||y|}e^{\frac{1}{2}\tilde{v}_k(y)}ds_y\\
&+\int_{\{y\in \partial^0 B^+_{\frac{1}{2}\lambda_k^{-1}}\ |\  |y-x|\leq \frac{1}{2}|x|\}}\log\frac{|x-y|}{|x||y|}e^{\frac{1}{2}\tilde{v}_k(y)}ds_y+ \int_{\{\partial^0 B^+_{\frac{1}{2}\lambda_k^{-1}}\ |\  |y|\leq R_k\}}\log\frac{|x-y|}{|x||y|}e^{\frac{1}{2}\tilde{v}_k(y)}ds_y\\
=& \mathbf{II}_1+\mathbf{II}_2+\mathbf{II}_3+\mathbf{II}_4.
\end{align*}

For $\mathbf{I}_1$ and $\mathbf{II}_1$, since $|y|\geq 2|x|$, it holds $$\frac{1}{2|x|}\leq \frac{1}{|x|}-\frac{1}{|y|}\leq \frac{|x-y|}{|x||y|}\leq \frac{1}{|x|}+\frac{1}{|y|}\leq \frac{3}{2|x|},$$ which implies $$|\mathbf{I}_1|\leq C\log |x|\int_{ B_{\frac{1}{2}\lambda_k^{-1}}\setminus B_{R_k}}e^{\tilde{v}_k(y)}dy =o(1)\log|x|$$ and  $$|\mathbf{II}_1|\leq C\log |x|\int_{ y\in \partial^0B^+_{\frac{1}{2}\lambda_k^{-1}},\  |y|\geq R_k}e^{\frac{1}{2}\tilde{v}_k(y)}ds_y =o(1)\log|x|.$$

For $\mathbf{I}_2$ and $\mathbf{II}_2$, since $|y|\leq 2|x|$ and $|y-x|\geq \frac{1}{2}|x|$, it holds $$\frac{1}{2|y|}\leq \frac{|x-y|}{|x||y|}\leq \frac{1}{|x|}+\frac{1}{|y|}\leq \frac{3}{|y|},$$
which implies
$$\left|\log \frac{|x-y|}{|x||y|}\right|\leq C\log|y|\leq C\log|x|.$$
Hence,
$$|\mathbf{I}_2|\leq C\log |x|\int_{ B_{\frac{1}{2}\lambda_k^{-1}}\setminus B_{R_k}}e^{\tilde{v}_k(y)}dy =o(1)\log|x|$$
and
$$|\mathbf{II}_2|\leq C\log |x|\int_{y\in\partial^0 B^+_{\frac{1}{2}\lambda_k^{-1}},\ \ |y|\geq R_k}e^{\frac{1}{2}\tilde{v}_k(y)}ds_y =o(1)\log|x|.$$

For $\mathbf{I}_3$ and $\mathbf{II}_3$, since $|y-x|\leq \frac{1}{2}|x|$, it holds that $\frac{1}{2}|x|\leq |y|\leq \frac{3}{2}|x|$. Hence we have
\begin{align*}
\mathbf{I}_3&=\int_{\{ |y-x|\leq \frac{1}{2}|x|\}}\log |x-y|e^{v_k(y)}dy-\int_{\{ |y-x|\leq \frac{1}{2}|x|\}}(\log |x|+\log |y|) e^{v_k(y)}dy\\
&=\int_{\{ |y-x|\leq \frac{1}{2}|x|\}}\log |x-y|e^{v_k(y)}dy+o(1)\log |x|\\
&=\int_{\{ |y-x|\leq \frac{1}{2}|x|\}}\log |x-y|\frac{o(1)}{|y|^2}dy+o(1)\log |x|\\
&=\frac{o(1)}{|x|^2}\int_{\{ |y-x|\leq \frac{1}{2}|x|\}}\log |x-y|dy+o(1)\log |x|=o(1)\log |x|
\end{align*} where we have used the fact that $v_k(y)+2\log |y|\leq -N_k$ which follows from Lemma \ref{lem:fast-decay}.

Similarly,
\begin{align*}
\mathbf{II}_3&=\int_{\{y\in\partial^0B^+_{\frac{1}{2}\lambda_k^{-1}},\ |y-x|\leq \frac{1}{2}|x|\}}\log |x-y|e^{\frac{1}{2}\tilde{v}_k(y)}ds_y-\int_{\{y\in\partial^0B^+_{\frac{1}{2}\lambda_k^{-1}},\ |y-x|\leq \frac{1}{2}|x|\}}(\log |x|+\log |y|) e^{\frac{1}{2}\tilde{v}_k(y)}ds_y\\
&=\int_{\{y\in\partial^0B^+_{\frac{1}{2}\lambda_k^{-1}},\ |y-x|\leq \frac{1}{2}|x|\}}\log |x-y|e^{\frac{1}{2}\tilde{v}_k(y)}ds_y+o(1)\log |x|\\
&=\int_{\{y\in\partial^0B^+_{\frac{1}{2}\lambda_k^{-1}},\ |y-x|\leq \frac{1}{2}|x|\}}\log |x-y|\frac{o(1)}{|y|}ds_y+o(1)\log |x|\\
&=\frac{o(1)}{|x|}\int_{\{y\in\partial^0B^+_{\frac{1}{2}\lambda_k^{-1}},\ |y-x|\leq \frac{1}{2}|x|\}}\log |x-y|ds_y+o(1)\log |x|=o(1)\log |x|.
\end{align*}

For $\mathbf{I}_4$ and $\mathbf{II}_4$, since $|y|\leq R_k$ and $|x|\geq 2R_k$, we have $$\frac{1}{2|y|}\leq \frac{1}{|y|}-\frac{1}{|x|}\leq \frac{|x-y|}{|x||y|}\leq \frac{1}{|x|}+\frac{1}{|y|}\leq \frac{3}{2|y|}.$$ Since $\tilde{v}_k(x)\leq 0$,  we get
\begin{align*}
|\mathbf{I}_4|&\leq C\int_{\{ |y|\leq R_k\}}\big|\log|y|\big|e^{\tilde{v}_k(y)}dy\\
&\leq C\int_{\{ |y|\leq 2\}}\big|\log|y|\big|e^{\tilde{v}_k(y)}dy+C\int_{\{ 2\leq |y|\leq R_k\}}\log|y|e^{\tilde{v}_k(y)}dy\\
&\leq C\int_{\{ |y|\leq 2\}}\big|\log|y|\big|dy+C\int_{\{ 2\leq |y|\leq R_k\}}\log|y|\frac{1}{|y|^4}dy\leq C
\end{align*} and
\begin{align*}
|\mathbf{II}_4|&\leq C\int_{\{y\in\partial^0B^+_{\frac{1}{2}\lambda_k^{-1}}, |y|\leq R_k\}}\big|\log|y|\big|e^{\tilde{v}_k(y)}ds_y\\
&\leq C\int_{\{y\in\partial^0B^+_{\frac{1}{2}\lambda_k^{-1}}, |y|\leq 2\}}\big|\log|y|\big|e^{\tilde{v}_k(y)}ds_y+C\int_{\{y\in\partial^0B^+_{\frac{1}{2}\lambda_k^{-1}}, 2\leq |y|\leq R_k\}}\log|y|e^{\tilde{v}_k(y)}ds_y\\
&\leq C\int_{\{y\in\partial^0B^+_{\frac{1}{2}\lambda_k^{-1}}, |y|\leq 2\}}\big|\log|y|\big|ds_y+C\int_{\{y\in\partial^0B^+_{\frac{1}{2}\lambda_k^{-1}}, 2\leq |y|\leq R_k\}}\log|y|\frac{1}{|y|^4}ds_y\leq C.
\end{align*}
 As a result, \eqref{equ:08} and \eqref{equ:09} hold true. Therefore, by \eqref{equ:07}, we get $$w_k(x)+\phi_k(x)-w_k(0)-\phi_k(0)=-\alpha\Big(\frac{1}{2}\lambda_k^{-1}\Big)\log|x|+o(1)\log|x|+O(1),\ \ R_k\leq |x|\leq\frac{1}{2}\lambda_k^{-1},$$ which yields \eqref{ineq:09} since $$\alpha\Big(\frac{1}{2}\lambda_k^{-1}\Big)=4+o(1)$$
 by Lemma \ref{lem:01}. We proved the lemma.
\end{proof}

\

With the help of Lemma \ref{lem:02}, we get a pointwise  estimate of first derivative as follows.
\begin{lem}\label{lem:03}
We have
\begin{align*}
&\nabla \tilde{v}_k(x)-\nabla \int_{B_{\frac{1}{2}\lambda_k^{-1}}}H(2\lambda_kx,2\lambda_ky)e^{\tilde{v}(y)}dy\\&=-\frac{x}{|x|^2}\alpha(|x|)+\frac{O(1)}{|x|^{5/4}}+O(1)\frac{|\log dist(x,\partial\R^2_+)|}{|x|^{3/2}}+o(1)\lambda_k,\ \ x\in B^+_{\frac{3}{8}\lambda_k^{-1}}\setminus B^+_{R_k}
\end{align*} and $$\nabla \tilde{v}_k(x)=-\frac{x}{|x|^2}\alpha(|x|)+\frac{O(1)}{|x|^{5/4}}+O(1)\frac{|\log dist(x,\partial\R^2_+)|}{|x|^{3/2}}+O(1)\lambda_k,\ \ x\in B^+_{\frac{3}{8}\lambda_k^{-1}}\setminus B^+_{R_k}.$$
\end{lem}

%

\begin{proof}
We just need to show the first conclusion, since the second one follows from the fact $$\nabla \int_{B_{\frac{1}{2}\lambda_k^{-1}}}H(2\lambda_kx,2\lambda_ky)e^{\tilde{v}(y)}dy=O(1)\lambda_k,\ \ x\in B^+_{\frac{3}{8}\lambda_k^{-1}}\setminus B^+_{R_k}.$$

We use the notations as in Lemma \ref{lem:02}. Let $w_k(x),\ \phi_k(x),\ \eta_k(x)$ be defined as before. Extending $\eta_k$ evenly and by using Green's formula, we have $$\eta_k(x)-\eta_k(0)=\int_{\partial B_{\frac{1}{2}\lambda_k^{-1}}}\frac{\partial}{\partial r}\left(G(2\lambda_kx,2\lambda_ky)\right)\left(\eta_k(y)-\eta_k(0)\right)ds_y,$$ which implies
\begin{equation}
\nabla\eta_k(x)=o(1)\lambda_k,\ \ \forall |x|\leq\frac{3}{8}\lambda^{-1}_k.
\end{equation}

A direct computation yields
\begin{align*}
&\nabla w_k(x)+\nabla\phi_k(x)-\nabla \int_{B_{\frac{1}{2}\lambda_k^{-1}}}H(\lambda_kx,\lambda_ky)e^{\tilde{v}(y)}dy\\
=&-\frac{1}{2\pi}\int_{B_{\frac{1}{2}\lambda_k^{-1}}}\frac{x-y}{|x-y|^2}e^{\tilde{v}_k(y)}dy -\frac{1}{\pi}\int_{\partial^0 B^+_{\frac{1}{2}\lambda_k^{-1}}}\frac{x-y}{|x-y|^2}e^{\frac{1}{2}\tilde{v}_k(y)}ds_y\\
=&-\frac{1}{2\pi}\int_{\{y\in B_{\frac{1}{2}\lambda_k^{-1}}, |y|\geq 2|x|\}}\frac{x-y}{|x-y|^2}e^{\tilde{v}_k(y)}dy -\frac{1}{\pi}\int_{\{y\in \partial^0B^+_{\frac{1}{2}\lambda_k^{-1}}, |y|\geq 2|x|\}}\frac{x-y}{|x-y|^2}e^{\frac{1}{2}\tilde{v}_k(y)}ds_y \\
&-\frac{1}{2\pi}\int_{\{ |y|\leq 2|x|,\ |y-x|\geq \frac{1}{2}|x|,\ |y|\geq |x|^s\}}\frac{x-y}{|x-y|^2}e^{\tilde{v}_k(y)}dy-\frac{1}{\pi}\int_{\{y\in \partial^0B^+_{\frac{1}{2}\lambda_k^{-1}}, |y|\leq 2|x|,\ |y-x|\geq \frac{1}{2}|x|,\ |y|\geq |x|^s\}}\frac{x-y}{|x-y|^2}e^{\frac{1}{2}\tilde{v}_k(y)}ds_y\\
&-\frac{1}{2\pi}\int_{\{ |y-x|\leq \frac{|x|}{2}\}}\frac{x-y}{|x-y|^2}e^{\tilde{v}_k(y)}dy -\frac{1}{\pi}\int_{\{y\in \partial^0B^+_{\frac{1}{2}\lambda_k^{-1}}, |y-x|\leq \frac{1}{2}|x|\}}\frac{x-y}{|x-y|^2}e^{\frac{1}{2}\tilde{v}_k(y)}ds_y\\
&-\frac{1}{2\pi}\int_{\{ |y|\leq |x^s|\}}\frac{x-y}{|x-y|^2}e^{\tilde{v}_k(y)}dy-\frac{1}{\pi}\int_{\{y\in \partial^0B^+_{\frac{1}{2}\lambda_k^{-1}}, |y|\leq |x|^s\}}\frac{x-y}{|x-y|^2}e^{\frac{1}{2}\tilde{v}_k(y)}ds_y\\
:=&\mathbf{III}_1+\mathbf{III}_2+\mathbf{III}_3+\mathbf{III}_4,
\end{align*} where $s\in (0,1)$ is a constant which will be chosen later.

For $\mathbf{III}_1$, since $|y|\geq 2|x|$, by Lemma \ref{lem:02}, we have $$\left| \int_{\{y\in B_{\frac{1}{2}\lambda_k^{-1}}, |y|\geq 2|x|\}}\frac{x-y}{|x-y|^2}e^{\tilde{v}_k(y)}dy \right|\leq \frac{C}{|x|}\int_{\{y\in B_{\frac{1}{2}\lambda_k^{-1}}, |y|\geq 2|x|\}} \frac{1}{|y|^{4-2\delta}}dy\leq \frac{C}{|x|^{3-2\delta}}$$ and $$\left| \int_{\{y\in \partial^0B^+_{\frac{1}{2}\lambda_k^{-1}}, |y|\geq 2|x|\}}\frac{x-y}{|x-y|^2}e^{\frac{1}{2}\tilde{v}_k(y)}dy\right|\leq \frac{C}{|x|}\int_{\{y\in \partial^0B^+_{\frac{1}{2}\lambda_k^{-1}}, |y|\geq 2|x|\}}\frac{1}{|y|^{2-\delta}}ds_y\leq \frac{C}{|x|^{2-\delta}},$$ which implies $$\mathbf{III}_1=\frac{O(1)}{|x|^{2-\delta}}.$$

For $\mathbf{III}_2$, since $|y-x|\geq\frac{1}{2}|x|$, by Lemma \ref{lem:02}, we see $$\left|\int_{\{ |y|\leq 2|x|,\ |y-x|\geq \frac{1}{2}|x|,\ |y|\geq |x|^s\}}\frac{x-y}{|x-y|^2}e^{\tilde{v}_k(y)}dy \right|\leq \frac{C}{|x|}\int_{\{  |x|^s\leq |y|\leq 2|x|\}}\frac{1}{|y|^{4-2\delta}}dy \leq \frac{C}{|x|^{1+(2-2\delta)s}}$$ and $$\left|\int_{\{y\in \partial^0B^+_{\frac{1}{2}\lambda_k^{-1}}, |y|\leq 2|x|,\ |y-x|\geq \frac{1}{2}|x|,\ |y|\geq |x|^s\}}\frac{x-y}{|x-y|^2}e^{\frac{1}{2}\tilde{v}_k(y)}ds_y \right|\leq \frac{C}{|x|}\int_{\{y\in \partial^0B^+_{\frac{1}{2}\lambda_k^{-1}},  |x|^s\leq |y|\leq 2|x|\}}\frac{1}{|y|^{2-\delta}}ds_y \leq \frac{C}{|x|^{1+(1-\delta)s}},$$ which implies $$\mathbf{III}_2=\frac{O(1)}{|x|^{1+(1-\delta)s}}.$$

For $\mathbf{III}_3$, since $|y-x|\leq\frac{1}{2}|x|$, it holds $\frac{1}{2}|x|\leq |y|\leq \frac{3}{2}|x|$. Then by Lemma \ref{lem:02}, we get $$\left|\int_{\{  |y-x|\leq \frac{1}{2}|x|\}}\frac{x-y}{|x-y|^2}e^{\tilde{v}_k(y)}dy \right|\leq \frac{C}{|x|^{4-2\delta}}\int_{\{   |y-x|\leq \frac{1}{2}|x|\}}\frac{1}{|y-x|}dy \leq \frac{C}{|x|^{3-2\delta}}$$ and
\begin{align*}
\left|\int_{\{y\in \partial^0B^+_{\frac{1}{2}\lambda_k^{-1}},  |y-x|\leq \frac{1}{2}|x|\}}\frac{x-y}{|x-y|^2}e^{\frac{1}{2}\tilde{v}_k(y)}ds_y \right|&\leq \frac{C}{|x|^{2-\delta}}\int_{\{y\in \partial^0B^+_{\frac{1}{2}\lambda_k^{-1}},   |y-x|\leq \frac{1}{2}|x|\}}\frac{1}{|y-x|}ds_y\\&\leq \frac{C}{|x|^{2-\delta}}\left( |\log dist(x,\partial\R^2_+)|+\log |x| \right),
\end{align*} which implies $$\mathbf{III}_3=\frac{O(1)\log|x|}{|x|^{2-\delta}}+O(1)\frac{|\log dist(x,\partial\R^2_+)|}{|x|^{2-\delta}}.$$

For $\mathbf{III}_4$, since $$\frac{x-y}{|x-y|^2}-\frac{x}{|x|^2}=\frac{O(1)}{|x|^{2-s}},$$ we find
\begin{align*}
-\frac{1}{2\pi}\int_{\{ |y|\leq |x|^s\}}\frac{x-y}{|x-y|^2}e^{\tilde{v}_k(y)}dy&=-\frac{1}{2\pi}\frac{x}{|x|^2}\int_{\{ |y|\leq |x|^s\}}e^{\tilde{v}_k(y)}dy+\frac{O(1)}{|x|^{2-s}}\\
&=-\frac{1}{2\pi}\frac{x}{|x|^2}\left(\int_{\{ |y|\leq |x|\}}e^{\tilde{v}_k(y)}dy-\int_{\{ |x|^s\leq |y|\leq |x|\}}e^{\tilde{v}_k(y)}dy\right)+\frac{O(1)}{|x|^{2-s}}\\
&=-\frac{1}{2\pi}\frac{x}{|x|^2}\int_{\{ |y|\leq |x|\}}e^{\tilde{v}_k(y)}dy+\frac{O(1)}{|x|^{1+(2-2\delta)s}}+\frac{O(1)}{|x|^{2-s}}
\end{align*} and
\begin{align*}
&-\frac{1}{\pi}\int_{\{y\in \partial^0B^+_{\frac{1}{2}\lambda_k^{-1}},  |y|\leq |x|^s\}}\frac{x-y}{|x-y|^2}e^{\frac{1}{2}\tilde{v}_k(y)}ds_y\\&=-\frac{1}{\pi}\frac{x}{|x|^2}\int_{\{y\in \partial^0B^+_{\frac{1}{2}\lambda_k^{-1}},  |y|\leq |x|^s\}}e^{\frac{1}{2}\tilde{v}_k(y)}ds_y+\frac{O(1)}{|x|^{2-s}}\\
&=-\frac{1}{\pi}\frac{x}{|x|^2}\left(\int_{\{y\in \partial^0B^+_{\frac{1}{2}\lambda_k^{-1}},  |y|\leq |x|\}}e^{\frac{1}{2}\tilde{v}_k(y)}ds_y-\int_{\{y\in \partial^0B^+_{\frac{1}{2}\lambda_k^{-1}},  |x|^s\leq |y|\leq |x|\}}e^{\frac{1}{2}\tilde{v}_k(y)}ds_y\right)+\frac{O(1)}{|x|^{2-s}}\\
&=-\frac{1}{\pi}\frac{x}{|x|^2}\int_{\{y\in \partial^0B^+_{\frac{1}{2}\lambda_k^{-1}},  |y|\leq |x|\}}e^{\frac{1}{2}\tilde{v}_k(y)}ds_y+\frac{O(1)}{|x|^{1+(1-\delta)s}}+\frac{O(1)}{|x|^{2-s}},
\end{align*} where we have used the fact that $$\int_{\{ |x|^s\leq |y|\leq |x|\}}e^{\tilde{v}_k(y)}dy\leq C\int_{\{ |x|^s\leq |y|\leq |x|\}}\frac{1}{|y|^{4-2\delta}}dy=\frac{O(1)}{|x|^{(2-2\delta)s}}$$ and $$\int_{\{y\in \partial^0B^+_{\frac{1}{2}\lambda_k^{-1}}, |x|^s\leq |y|\leq |x|\}}e^{\frac{1}{2}\tilde{v}_k(y)}ds_y\leq C\int_{\{y\in \partial^0B^+_{\frac{1}{2}\lambda_k^{-1}}, |x|^s\leq |y|\leq |x|\}}\frac{1}{|y|^{2-\delta}}ds_y=\frac{O(1)}{|x|^{(1-\delta)s}}.$$
Thus, $$\mathbf{III}_4=-\frac{1}{2\pi}\frac{x}{|x|^2}\int_{\{ |y|\leq |x|\}}e^{\tilde{v}_k(y)}dy -\frac{1}{\pi}\frac{x}{|x|^2}\int_{\{y\in \partial^0B^+_{\frac{1}{2}\lambda_k^{-1}},  |y|\leq |x|\}}e^{\frac{1}{2}\tilde{v}_k(y)}ds_y +\frac{O(1)}{|x|^{1+(1-\delta)s}}+\frac{O(1)}{|x|^{2-s}}.$$

Combining these together, we get
\begin{align*}
&\nabla w_k(x)+\nabla\phi_k(x)-\nabla \int_{B_{\frac{1}{2}\lambda_k^{-1}}}H(\lambda_kx,\lambda_ky)e^{\tilde{v}(y)}dy\\
=&-\frac{1}{2\pi}\frac{x}{|x|^2}\int_{\{ |y|\leq |x|\}}e^{\tilde{v}_k(y)}dy -\frac{1}{\pi}\frac{x}{|x|^2}\int_{\{y\in \partial^0B^+_{\frac{1}{2}\lambda_k^{-1}},  |y|\leq |x|\}}e^{\frac{1}{2}\tilde{v}_k(y)}ds_y\\
& +\frac{O(1)}{|x|^{1+(1-\delta)s}}+\frac{O(1)}{|x|^{2-s}}+\frac{O(1)\log|x|}{|x|^{2-\delta}}+O(1)\frac{|\log dist(x,\partial\R^2_+)|}{|x|^{2-\delta}}.
\end{align*}
Taking $\delta=s=\frac{1}{2}$, we deduce
\begin{equation*}
\nabla w_k(x)+\nabla\phi_k(x)-\nabla \int_{B_{\frac{1}{2}\lambda_k^{-1}}}H(\lambda_kx,\lambda_ky)e^{\tilde{v}(y)}dy=-\frac{x}{|x|^2}\alpha(|x|)+\frac{O(1)}{|x|^{5/4}}+O(1)\frac{|\log dist(x,\partial\R^2_+)|}{|x|^{3/2}},
\end{equation*} which implies
\begin{equation*}
\nabla \tilde{v}_k(x)-\nabla \int_{B_{\frac{1}{2}\lambda_k^{-1}}}H(\lambda_kx,\lambda_ky)e^{\tilde{v}(y)}dy=-\frac{x}{|x|^2}\alpha(|x|)+\frac{O(1)}{|x|^{5/4}}+O(1)\frac{|\log dist(x,\partial\R^2_+)|}{|x|^{3/2}}+o(1)\lambda_k.
\end{equation*}
\end{proof}

\

As a corollary of above lemma, we can control the tangential oscillation as follows.
\begin{cor}\label{cor:01}
We have $$\sup_{R_k\leq t\leq\frac{3}{8}\lambda_k^{-1}}osc_{\partial^+B_t^+}\tilde{v}_k=o(1).$$
\end{cor}
\begin{proof}
For any $R_k\leq t\leq \lambda_k^{-\frac{1}{2}}$ and any two points $(t,\theta_1), (t,\theta_2)\in\partial^+B^+_t$, by Lemma \ref{lem:03}, it holds
\begin{align*}
|\tilde{v}_k(t,\theta_1)-\tilde{v}_k(t,\theta_2)|&=\left|\int_{\theta_2}^{\theta_1}\frac{\partial \tilde{v}_k(t,\theta)}{\partial\theta}d\theta\right|\\
&\leq \int_{\partial^+B_t^+}\left|\frac{1}{t}\frac{\partial \tilde{v}_k(y)}{\partial\theta}\right| ds_y\\
&\leq C\int_{\partial^+B_t^+}\left(\frac{1}{t^{5/4}}+\frac{|\log dist(y,\partial\R^2_+)|}{t^{3/2}}+\lambda_k\right) ds_y\\
&= C\int_{\partial^+B_t^+}\frac{|\log dist(y,\partial\R^2_+)|}{t^{3/2}} ds_y+o(1).
\end{align*}
Taking the coordinate notation $y=(y^1,y^2)$ and  noting that
\begin{align}\label{ineq:10}
\int_{\partial^+B^+_t}|\log dist(y,\partial\R^2_+)|ds_y&=\int_{\partial^+B^+_t}|\log|y^2||ds_y\notag\\
&=\int_{y\in \partial^+B^+_t,|y^2|\leq 1}|\log|y^2||ds_y+\int_{y\in \partial^+B^+_t,|y^2|\geq 1}|\log|y^2||ds_y\notag\\
&\leq C\int_0^1|\log y^2|dy^2+C\int_{y\in \partial^+B^+_t}\log tds_y\leq Ct\log t,
\end{align} we get
\begin{equation}\label{ineq:14}
\tilde{v}_k(t,\theta_1)-\tilde{v}_k(t,\theta_2)=o(1),\ \ R_k\leq t\leq \lambda_k^{-\frac{1}{2}}.
\end{equation}

For any $\lambda_k^{-\frac{1}{2}}\leq t\leq \frac{3}{8}\lambda_k^{-1}$ and any two points $z_1=(t,\theta_1), z_2= (t,\theta_2)\in\partial^+B^+_t$, set $$F(x)=\int_{B_{\frac{1}{2}\lambda_k^{-1}}}H(2\lambda_kx,2\lambda_ky)e^{\tilde{v}(y)}dy.$$
On one hand,  by Lemma \ref{lem:03}, we have
\begin{align}\label{ineq:12}
|\tilde{v}_k(t,\theta_1)-F(z_1)-\tilde{v}_k(t,\theta_2)+F(z_2)|&=\left|\int_{\theta_2}^{\theta_1}\frac{\partial }{\partial\theta}(\tilde{v}_k(t,\theta)-F(x))d\theta\right|\notag\\
&\leq \int_{\partial^+B_t^+}\left|\frac{1}{t}\frac{\partial }{\partial\theta}(\tilde{v}_k(x)-F(x))\right| ds_x\notag\\
&\leq C\int_{\partial^+B_t^+}\left(\frac{1}{t^{5/4}}+\frac{|\log dist(x,\partial\R^2_+)|}{t^{3/2}}+o(1)\lambda_k\right) ds_x\notag\\
&= C\int_{\partial^+B_t^+}\frac{|\log dist(x,\partial\R^2_+)|}{t^{3/2}} ds_x+o(1)=o(1).
\end{align}

On the other hand, by the formula of Green's function for the unit ball, we know $$H(x,y)=\frac{1}{2\pi}\log \left||x|\Big(y-\frac{x}{|x|^2}\Big)\right|.$$
Hence, we see
\begin{align}\label{ineq:13}
|F(z_1)-F(z_2)|&\leq \left| \int_{B_{R_k}}\left(H(2\lambda_kz_1,2\lambda_ky) -H(2\lambda_kz_2,2\lambda_ky)\right)e^{\tilde{v}(y)}dy\right|+ C\int_{B_{\frac{1}{2}\lambda_k^{-1}}\setminus B_{R_k}}e^{\tilde{v}(y)}dy\notag\\
&=\frac{1}{2\pi}\left| \int_{B_{R_k}}\log\frac{|4\lambda_k^2t^2y-z_1|}{|4\lambda_k^2t^2y-z_2|}e^{\tilde{v}(y)}dy\right|+ C\int_{B_{\frac{1}{2}\lambda_k^{-1}}\setminus B_{R_k}}e^{\tilde{v}(y)}dy\notag\\
&\leq \frac{1}{2\pi} \int_{B_{R_k}}\left|\log\frac{|4\lambda_k^2t^2y-z_1|}{|4\lambda_k^2t^2y-z_2|}\right|e^{\tilde{v}(y)}dy+ C\int_{B_{\frac{1}{2}\lambda_k^{-1}}\setminus B_{R_k}}\frac{1}{|y|^{4-2\delta}}dy=o(1),
\end{align}
 where we have used Lemma \ref{lem:02} and the following fact that
 $$\left|\log\frac{|4\lambda_k^2t^2y-z_1|}{|4\lambda_k^2t^2y-z_2|}\right|=o(1)$$
 as $k\to\infty$ since $\lambda_k^{-\frac{1}{2}}\leq t\leq \frac{3}{8}\lambda_k^{-1}$,  $|4\lambda_k^2t^2y|\leq CR_k$ and $\lambda_kR_k=o(1)$.

By \eqref{ineq:12} and \eqref{ineq:13}, we get $$\tilde{v}_k(t,\theta_1)-\tilde{v}_k(t,\theta_2)=o(1),\ \ \lambda_k^{-\frac{1}{2}}\leq t\leq \frac{3}{8}\lambda_k^{-1}.$$
 Combining this with \eqref{ineq:14}, we proved the corollary.
\end{proof}

\

With the help of Lemma \ref{lem:03}, using Pohozaev's type identity, we get the following energy decay estimate.
\begin{lem}\label{lem:blowup-value}
There holds
\begin{eqnarray*}
 \alpha(|x|)=4+\dis\frac{O(1)}{|x|^{\frac{1}{4}}},\ \ \forall\, R_k\leq |x|\leq  \frac{3}{8}\lambda_k^{-1}.
 \end{eqnarray*}
\end{lem}
\begin{proof}
By the Pohozaev identity,  we have
\begin{align*}
&2\left( \int_{B^+_t}e^{\tilde{v}_k}dx+  \int_{\partial^0 B^+_t}e^{\frac{1}{2}\tilde{v}_k}ds_x\right) \\
 =&t\int_{\partial^+B^+_t}\left(\Big|\frac{\partial \tilde{v}_k}{\partial r}\Big|^2-\frac{1}{2}|\nabla \tilde{v}_k|^2+e^{\tilde{v}_k}\right)ds_x+2te^{\frac{1}{2}\tilde{v}_k(t,0)}+2te^{\frac{1}{2}\tilde{v}_k(-t,0)},\ \ \forall \,0<t\leq\frac{3}{8}\lambda_k^{-1}.
\end{align*}

Taking $t=|x|$ in above equality,  by  Lemma \ref{lem:03}, Lemma \ref{lem:02} and \eqref{ineq:10}, we find
\begin{align*}
2\pi\alpha(|x|)=\frac{1}{2}\pi\alpha^2(|x|)+\alpha(|x|)\frac{O(1)}{|x|^{\frac{1}{4}}}+\frac{O(1)}{|x|^{\frac{1}{2}}} +\frac{O(1)}{|x|^{2-2\delta}}+\frac{O(1)}{|x|^{1-\delta}}.
\end{align*}
Then taking $\delta=\frac{1}{2}$, it holds
\begin{equation}
\alpha(|x|)=4+\frac{O(1)}{|x|^{\frac{1}{4}}},\ \ \forall\, R_k\leq |x|\leq \frac{3}{8}\lambda_k^{-1}.
\end{equation}
\end{proof}

\

Using the above energy decay estimate, we can control the oscillation of radial part as follows.
\begin{prop}\label{prop:01}
We have $$\lim_{k\to\infty}\|v_k(x)-v(x)\|_{C^0(B_{\frac{3}{8}\lambda_k^{-1}}\cap\Omega_k)}=0.$$
\end{prop}
\begin{proof}
By \eqref{equ:11}, one can see that the proposition follows from the following estimate
 $$\lim_{k\to 0}osc_{R_k\leq |x|\leq \frac{3}{8}\lambda_k^{-1},x\in\Omega_k}(v_k(x)-v(x))=0.$$
 Noting that $$v(x)+4\log|x|=o(1),\ \ \forall |x|\geq R_k,$$ we just need to show that
\begin{equation*}
\lim_{k\to 0}osc_{R_k\leq |x|\leq \frac{3}{8}\lambda_k^{-1},x\in\Omega_k}(v_k(x)+4\log|x|)=0,
\end{equation*} which is equivalent to
\begin{equation}\label{equ:03}
\lim_{k\to 0}osc_{x\in B^+_{\frac{3}{8}\lambda^{-1}_k}\setminus B^+_{R_k}}(\tilde{v}_k(x)+4\log|x|)=0,
\end{equation} where $\tilde{v}_k$ was defined by \eqref{equ:def-1}.

In fact, for any two points $p_1=(t_1,\theta_1),p_2=(t_2,\theta_2)\in B^+_{\frac{3}{8}\lambda^{-1}_k}\setminus B^+_{R_k}$ (without loss of generality, we assume  $t_1\leq t_2$), there holds
\begin{align}\label{ineq:11}
&\left|\tilde{v}_k(p_1)+4\log|p_1|-\tilde{v}_k(p_2)-4\log|p_2|\right|\notag\\
\leq &\left|\tilde{v}^*_k(t_1)+4\log|p_1|-\tilde{v}^*_k(t_2)-4\log|p_2|\right|  + osc_{\partial^+B^+_{t_1}}\tilde{v}_k+osc_{\partial^+B^+_{t_2}}\tilde{v}_k,
\end{align} where $\tilde{v}_k^*$ was defined by \eqref{equ:def-3}.

It follows from \eqref{equ:02} and Lemma \ref{lem:blowup-value} that $$\frac{d \tilde{v}_k^*}{dr}=-\frac{4}{r}+\frac{O(1)}{r^{5/4}},$$
which implies
\begin{align*}
\left|\tilde{v}^*_k(t_1)+4\log|p_1|-\tilde{v}^*_k(t_2)-4\log|p_2|\right| &=\left|\int_{t_1}^{t_2}\frac{d}{dr}(\tilde{v}^*_k(r)+4\log r)dr\right|\\
&\leq C\int_{t_1}^{t_2}\frac{1}{r^{5/4}}dr\leq \frac{C}{R_k^{1/4}}=o(1).
\end{align*}
Combining this with \eqref{ineq:11} and Corollary \ref{cor:01}, we get \eqref{equ:03}.
\end{proof}

\

\begin{proof}[\textbf{Proof of Theorem \ref{thm:02}}]
By Proposition \ref{prop:01}, we need to show that
\begin{equation*}
\lim_{k\to\infty}\|v_k(x)-v(x)\|_{C^0(\Omega_k \setminus B_{\frac{3}{8}\lambda_k^{-1}})}=0,
\end{equation*} which is equivalent to
\begin{equation}\label{equ:17}
\lim_{k\to\infty} osc_{\Omega_k \setminus B_{\frac{3}{8}\lambda_k^{-1}}}\left(v_k(x)+4\log |x|\right)=0.
\end{equation}

Noting that $\Omega_k \setminus B_{\frac{3}{8}\lambda_k^{-1}}(0)\subset \Omega_k \setminus B_{\frac{1}{4}\lambda_k^{-1}}(-\frac{x_k}{\lambda_k})$, one can see that \eqref{equ:17} is a consequence of the following fact
\begin{equation}\label{equ:18}
\lim_{k\to\infty}osc_{B^+_{\lambda_k^{-1}}(\bar{x}_k) \setminus B^+_{\frac{1}{4}\lambda_k^{-1}}(\bar{x}_k)}\big(\tilde{v}_k(x)+4\log|x|\big)=0,
\end{equation} where $\bar{x}_k:=-\frac{(x^1_k,0)}{\lambda_k}$ and $\tilde{v}_k$ was defined by \eqref{equ:def-1}.

Let $w^1_k$ be the solution of
\begin{align*}
\begin{cases}
-\Delta w^1_k=-\Delta \tilde{v}_k,\ \ &in\ B^+_{\lambda_k^{-1}}(\bar{x}_k) \setminus B^+_{\frac{1}{4}\lambda_k^{-1}}(\bar{x}_k),\\
\dis\frac{\partial w^1_k}{\partial \vec{n}}=0,\ \ &on\ \partial^0 B^+_{\lambda_k^{-1}}(\bar{x}_k) \setminus \partial^0 B^+_{\frac{1}{4}\lambda_k^{-1}}(\bar{x}_k),\\
w^1_k=0,\ \ &on\ \partial^+B^+_{\lambda_k^{-1}}(\bar{x}_k) \cup \partial^+ B^+_{\frac{1}{4}\lambda_k^{-1}}(\bar{x}_k).
\end{cases}
\end{align*}
Extending $w^1_k$ evenly, by standard elliptic theory and Lemma \ref{lem:02}, we can check
\begin{align*}
\|w^1_k\|_{C^0(B^+_{\lambda_k^{-1}}(\bar{x}_k) \setminus B^+_{\frac{1}{4}\lambda_k^{-1}}(\bar{x}_k))}&\leq C\|\lambda_k^{-2}e^{\tilde{v}_k(\lambda_k^{-1}(x+\lambda_k\bar{x}_k))}\|_{L^2(B^+_1(0)\setminus B^+_{\frac{1}{4}}(0))}\\
&\leq C\|\lambda_k^{-2}\frac{1}{|\lambda_k^{-1}(x+\lambda_k\bar{x}_k)|^{4-2\delta}}\|_{L^2(B^+_1(0)\setminus B^+_{\frac{1}{4}}(0))}\\
&\leq C\|\lambda_k^{2-2\delta}\frac{1}{|x-(x^1_k,0)|^{4-2\delta}}\|_{L^2(B^+_1(0)\setminus B^+_{\frac{1}{4}}(0))}\leq C\lambda_k^{2-2\delta}=o(1),
\end{align*} since $x_k=(x^1_k,x^2_k)\to 0$ as $k\to\infty$.

Denote $$w^2_k(x):=-\frac{1}{\pi}\int_{\partial^0B^+_{\lambda_k^{-1}}(\bar{x}_k)}\log|x-y|g(y)ds_y,\ \ x\in B^+_{\lambda_k^{-1}}(\bar{x}_k),$$ where $g(y)=0$ if $y\in\partial^0B^+_{\frac{1}{4}\lambda_k^{-1}}(\bar{x}_k)$ and $g(y)=e^{\frac{1}{2}\tilde{v}_k(y)}$ if $y\in \partial^0B^+_{\lambda_k^{-1}}(\bar{x}_k) \setminus \partial^0B^+_{\frac{1}{4}\lambda_k^{-1}}(\bar{x}_k)$.

It is easy to check that
\begin{align*}
\begin{cases}
-\Delta w^2_k=0,\ \ &in\ B^+_{\lambda_k^{-1}}(\bar{x}_k),\\
\dis\frac{\partial w^2_k}{\partial \vec{n}}=g,\ \ &on\ \partial^0 B^+_{\lambda_k^{-1}}(\bar{x}_k) .
\end{cases}
\end{align*} Moreover, we claim $$\|w^2_k\|_{C^0(B^+_{\lambda_k^{-1}}(\bar{x}_k))}=o(1).$$
In fact, for any $x\in B^+_{\lambda_k^{-1}}(\bar{x}_k) \setminus B^+_{\frac{1}{4}\lambda_k^{-1}}(\bar{x}_k)$, if $dist(x,\partial\R^2_+)>1$, then by Lemma \ref{lem:02}, we have
\begin{align*}
|w^2_k(x)|&\leq C\log\lambda_k^{-1}\int_{\partial^0B^+_{\lambda_k^{-1}}(\bar{x}_k)}|g(y)|ds_y\\
&=C\log\lambda_k^{-1}\int_{\partial^0 B^+_{\lambda_k^{-1}}(\bar{x}_k) \setminus \partial^0 B^+_{\frac{1}{4}\lambda_k^{-1}}(\bar{x}_k)}e^{\frac{1}{2}\tilde{v}_k(y)}ds_y\\
&\leq C\log\lambda_k^{-1}\int_{\partial^0 B^+_{\lambda_k^{-1}}(\bar{x}_k) \setminus \partial^0 B^+_{\frac{1}{4}\lambda_k^{-1}}(\bar{x}_k)}\frac{1}{|y|^{2-\delta}}ds_y\\
&=C\lambda_k^{1-\delta}\log\lambda_k^{-1}=o(1).
\end{align*}
If $dist(x,\partial^0\R^2_+)\leq 1$, from Lemma \ref{lem:02}, it holds
\begin{align*}
|w^2_k(x)|&=\left| \frac{1}{\pi}\int_{\partial^0B^+_{\lambda_k^{-1}}(\bar{x}_k)}\log|x-y|g(y)ds_y\right|\\
&\leq C\lambda_k^{2-\delta}\int_{\partial^0 B^+_{\lambda_k^{-1}}(\bar{x}_k) \setminus \partial^0 B^+_{\frac{1}{4}\lambda_k^{-1}}(\bar{x}_k)}\log|x-y|ds_y\\
&\leq C\lambda_k^{2-\delta}\int_{\{y\in\partial^0 B^+_{\lambda_k^{-1}}(\bar{x}_k) \setminus \partial^0 B^+_{\frac{1}{4}\lambda_k^{-1}}(\bar{x}_k),\ |y-x|\leq 1\}}\log|x-y|ds_y\\&\quad + C\lambda_k^{2-\delta}\int_{\{y\in\partial^0 B^+_{\lambda_k^{-1}}(\bar{x}_k) \setminus \partial^0 B^+_{\frac{1}{4}\lambda_k^{-1}}(\bar{x}_k),\ |y-x|\geq 1\}}\log|x-y|ds_y\\
&\leq C\lambda_k^{2-\delta}(1+\lambda_k^{-1}\log\lambda_k^{-1})=o(1),
\end{align*} which implies the claim.

Now let $w^3_k=\tilde{v}_k+4\log|x|-w^1_k(x)-w^2_k(x)$. We have
\begin{align*}
\begin{cases}
-\Delta w^3_k=0,\ \ &in\ B^+_{\lambda_k^{-1}}(\bar{x}_k) \setminus B^+_{\frac{1}{4}\lambda_k^{-1}}(\bar{x}_k),\\
\dis\frac{\partial w^3_k}{\partial \vec{n}}=0,\ \ &on\ \partial^0 B^+_{\lambda_k^{-1}}(\bar{x}_k) \setminus \partial^0 B^+_{\frac{1}{4}\lambda_k^{-1}}(\bar{x}_k),\\
w^3_k=\tilde{v}_k+4\log|x|-w^2_k(x),\ \ &on\ \partial^+B^+_{\lambda_k^{-1}}(\bar{x}_k) \cup \partial^+ B^+_{\frac{1}{4}\lambda_k^{-1}}(\bar{x}_k).
\end{cases}
\end{align*}
Extend $w^3_k$ evenly and note that
\begin{eqnarray*}
osc_{\partial^+B^+_{\lambda_k^{-1}}(\bar{x}_k) }(\tilde{v}_k+4\log|x|-w^2_k(x)) &\leq &osc_{\partial^+B^+}u_k(x)
+osc_{\partial^+B^+_{\lambda_k^{-1}}(\bar{x}_k) }\log|x|+osc_{\partial^+B^+_{\lambda_k^{-1}}(\bar{x}_k) }w^2_k(x)\\
&=&o(1)
\end{eqnarray*}
and
\begin{eqnarray*}
osc_{\partial^+B^+_{\frac{1}{4}\lambda_k^{-1}}(\bar{x}_k)}(\tilde{v}_k+4\log|x|-w^2_k(x))&\leq&  \big\|\tilde{v}_k(x)+4\log|x|\big\|_{C^0(\partial^+B^+_{\frac{1}{4}\lambda_k^{-1}}(\bar{x}_k))} +\|w^2_k(x)\|_{C^0(\partial^+B^+_{\frac{1}{4}\lambda_k^{-1}}(\bar{x}_k))}\\
&=&o(1),
\end{eqnarray*}
 where we have used the fact that $\big\|\tilde{v}_k(x)+4\log|x|\big\|_{C^0(\partial^+B^+_{\frac{1}{4}\lambda_k^{-1}}(\bar{x}_k))}=o(1)$, which follows from Proposition \ref{prop:01}.

 By maximal principle, we deduce $$osc_{B^+_{\lambda_k^{-1}}(\bar{x}_k) \setminus B^+_{\frac{1}{4}\lambda_k^{-1}}(\bar{x}_k)}w^3_k=o(1).$$

Combining  these together, we get \eqref{equ:18}.

Consequently, we complete  the proof.
\end{proof}

\


\providecommand{\bysame}{\leavevmode\hbox to3em{\hrulefill}\thinspace}
\providecommand{\MR}{\relax\ifhmode\unskip\space\fi MR }
\providecommand{\MRhref}[2]{%
  \href{http://www.ams.org/mathscinet-getitem?mr=#1}{#2}
}
\providecommand{\href}[2]{#2}

\end{document}